\documentclass[11pt]{amsart}

\usepackage{epsfig}

\usepackage{graphicx}
\usepackage{amscd}
\usepackage{amsmath}
\usepackage{amsfonts}
\usepackage{psfrag}
\usepackage{amssymb}
\usepackage{verbatim}
\usepackage{comment}

\newtheorem{theorem}{Theorem}[section]
\theoremstyle{plain}

\newtheorem{corollary}[theorem]{Corollary}

\newtheorem{defi}[theorem]{Definition}
\newtheorem{example}[theorem]{Example}

\newtheorem{lemma}[theorem]{Lemma}

\newtheorem{prop}[theorem]{Proposition}
\newtheorem{remark}[theorem]{Remark}

\numberwithin{equation}{section}

\newcommand{\1}{{\rm1\!\!1}}

\def\Xxi{{\mathfrak X}_\zeta}

\def\One{{1\!\!1}}

\def\Ok{{\mathcal O}}

\def\wtil{\widetilde}

\newcommand{\lam}{\lambda}

\newcommand{\om}{\omega}

\newcommand{\Sig}{\Sigma}

\newcommand{\sig}{\sigma}

\newcommand{\R}{{\mathbb R}}

\newcommand{\Z}{{\mathbb Z}}
\newcommand{\C}{{\mathbb C}}

\def\N{{\mathbb N}}

\newcommand{\Nat}{{\mathbb N}}

\def\A{{\mathcal A}}

\def\Rk{{\mathcal R}}
\def\Qk{{\mathcal Q}}

\def\Sf{{\sf S}}
\def\T{{\mathbb T}}

\def\be{\begin{equation}}
\def\ee{\end{equation}}

\newcommand{\eps}{{\varepsilon}}

\def\ov{\overline}

\def\a{a}

\def\ve1{\vec{1}}

\def\Ak{{\mathcal A}}

\def\Cf{{\sf C}}
\def\wt{\widetilde}
\def\what{\widehat}

\begin{document}

\title[Singular substitutions of constant length]{Singular substitutions of constant length}

\author{Artemi Berlinkov }
\address{Artemi Berlinkov, Department of Mathematics,
Bar-Ilan University, Ramat-Gan, Israel}
\email{artemi.berlinkov@biu.ac.il}

\author{Boris Solomyak }
\address{Boris Solomyak, Department of Mathematics,
Bar-Ilan University, Ramat-Gan, Israel}
\email{bsolom3@gmail.com}

\begin{abstract} We consider primitive aperiodic substitutions of constant length $q$ and prove that, in order to have a Lebesgue component in the spectrum of the associated dynamical system,  it is necessary that one of the eigenvalues of the substitution matrix equals $\sqrt{q}$ in absolute value. The proof is based on results of M. Queff\'elec, combined with estimates of the local dimension of the spectral measure at zero. 
\end{abstract}


\thanks{Both authors were supported in part by the Israel Science Foundation (grant 396/15).}
\thanks{A. Berlinkov was supported in part by the Center for Absorption in Science,  Ministry of Immigrant 
Absorption, State of Israel.}

\maketitle

\thispagestyle{empty}

\section{Introduction}\label{intro}
In this paper we focus on the question: when is the spectrum of a substitution dynamical system singular? Our main result is a new sufficient condition for singularity of constant length substitutions. 

In order to state the results, we first recall the background; see e.g.\ \cite{Queff,Siegel} for more details.
For $m\ge 2$ consider a finite alphabet $\A$ of size $m$ and the set $\A^+$  of nonempty words with letters in $\A$. 
A {\em substitution}  is a map $\zeta:\,\A\to \A^+$, extended to 
 $\A^+$ and $\A^{\N}$ by
concatenation. The {\em substitution space} $X_\zeta$ is
the set of bi-infinite sequences $x\in \A^\Z$ such that any word  in $x$
appears as a subword of $\zeta^n(\a)$ for some $\a\in \A$ and $n\in \N$. The {\em substitution dynamical system}  is the left
shift on $\A^\Z$ restricted to $X_\zeta$, denoted by $T_\zeta$.
The {\em substitution matrix} is defined by
$$
\Sf_\zeta(a,b):= \mbox{number of  symbols $a$ in $\zeta(b)$}.
$$
 The substitution is said to be {\em primitive} if $\Sf_\zeta^n$ has all entries strictly positive for some $n\in \Nat$.
Primitive substitution $\Z$-actions are minimal and uniquely ergodic; we denote the unique invariant Borel probability measure by $\mu$.
We always assume that the substitution is  {\em non-periodic}, which in the primitive case is equivalent to the space $X_\zeta$ being infinite. Furthermore, passing to a power of $\zeta$ if necessary, we can assume that there exists a letter $a$ such that $\zeta(a)$ starts with $a$. We then obtain a one-sided fixed point of the substitution:
\be \label{fixpt}
U=u_0u_1u_2\ldots = \lim_{n\to\infty} \zeta^n(a).
\ee

The length of a word $u$ is denoted by $|u|$. The substitution $\zeta$ is  of {\em constant length} $q$ if $|\zeta(a)|=q$ for all $a\in \A$, otherwise, it is of {\em non-constant length}.
Spectral properties of substitution dynamical systems have been studied extensively. These systems are never strongly mixing \cite{DK}, hence there is always a singular spectral component. 
Suppose that the substitution has constant length $q$. Then the group of eigenvalues is non-trivial and contains
the group of $q$-adic rationals; a complete description of the discrete spectral component was given by Dekking \cite{Dek1}. Since we need it later, we state it precisely.
The {\em height} $h$ of the substitution is defined as follows: let
$$
g_0 ={\rm gcd} \{k\ge 1:\ u_k=u_0\},
$$
where $U$ is the fixed point of the substitution (\ref{fixpt}); then define, for a substitution $\zeta$ of constant length $q$:
\be \label{height}
h=h(\zeta) = \max\{n\ge 1:\ (n,q)=1,\ n\ \mbox{divides}\ g_0\}.
\ee
Dekking \cite{Dek1} proved that the group of eigenvalues for $(X_\zeta,T_\zeta,\mu)$, where $\zeta$ is primitive aperiodic of constant length $q$, is $e(\Z(q)\times \Z/h\Z)$, where
$\Z(q)$ is the group of $q$-adic rationals and $e(t) = e^{2\pi i t}$. see also \cite{Martin} for earlier results in this direction. 
Furthermore, Dekking showed that it is always possible to reduce a substitution with $h>1$ to one with $h=1$, called ``pure base''. We show that our sufficient condition for singularity of the substitution can be used with substitutions of any height.

For a  substitution of height one, 
Dekking's coincidence condition \cite{Dek1}  gives an answer for when the spectrum is purely discrete. Thus it remains to analyze the continuous component, when it is present. The work of Queff\'elec \cite{Queff} contains a detailed and extensive study of this question; in particular, she expressed the maximal spectral type in terms of
generalized matrix Riesz products (which sometimes reduce to scalar generalized Riesz products). Not many general results on singularity of the spectrum are known.
Pure singular spectrum has been proved for the Thue-Morse substitution \cite{Kaku} and its generalizations \cite{Queff}, the so-called abelian bijective substitutions (see Section 6 for definitions). For bijective substitutions on two symbols a constructive proof of singularity was given by Baake, G\"ahler, and Grimm \cite{Baake}.

On the other hand, there are substitutions of constant length with a Lebesgue spectral component, such as the  Rudin-Shapiro substitution, see \cite{Queff}, and its generalizations, due to Frank \cite{Frank}, and more recently, by Chan and Grimm \cite{ChGr2}.
Recently Bartlett \cite{Bartlett} reworked (and in the case of one important example corrected) a part of Queff\'elec's theory, extended it to higher-dimensional block-substitutions of ``constant shape,'' and developed an algorithm for checking singularity. However, it proceeds by a ``case-by-case'' analysis and does not provide a general criterion. Our main result is as follows.

\begin{theorem} \label{th-main1}
Let $\zeta$ be a primitive aperiodic substitution of constant length $q$. If  the substitution matrix $\Sf_\zeta$ has no eigenvalue whose absolute value equals $\sqrt{q}$, then the maximal spectral type of the substitution measure-preserving system 
$(X_\zeta,T_\zeta,\mu)$ is singular.
\end{theorem}

\begin{remark} {\em
(1) The condition for singularity in our theorem is sufficient, but not necessary, see \cite{ChGr1} and Section~\ref{examples} below.

(2) The spectrum of substitution dynamical systems is closely related to diffraction spectrum of aperiodic structures, studied both by physicists and mathematicians. 
For a comprehensive account of this connection and extensive bibliography, see the recent book \cite{BaGr}. It contains, in particular, a constructive rigorous treatment of the spectra of Thue-Morse and Rudin-Shapiro substitutions.
The ``scaling exponent'' $1/2$ of the ``structure factor'' has been linked with absolutely continuous spectrum in the physics literature, see e.g.\ \cite{AGL,GL}, however, as far as we are aware, there has been no rigorous proof of singularity following this line.

(3) Higher-dimensional block substitutions and the corresponding  $\Z^d$ actions have also been studied, motivated in part by the connection with the theory of quasicrystals.
Without an attempt to give a comprehensive list of references, we mention the paper of Frank \cite{Frank2}, which, in particular contains examples of systems with singular spectrum, and the paper of Baake and Grimm \cite{BG}, who proved singularity for higher-dimensional 2-symbol bijective block-substitutions.
We expect that Theorem~\ref{th-main1} can be extended to higher dimensions using the results of Bartlett \cite{Bartlett} and Emme \cite{Emme}, but have not checked  the details.}
\end{remark}

The paper is organized as follows: in the next section we recall some of the background and outline the proof of the theorem. In Section~\ref{Queffelec} we recall the needed results of Queff\'elec and and bring to the form convenient for us.
In Section~\ref{purebase} we prove that the eigenvalues of the substitution matrix after reduction to substitution of height 1 are preserved except for 0  and roots of unity, therefore without loss of generality we can assume that the substitution has height 1.
In Section~\ref{dimension} we obtain results on the dimension of spectral measures at zero and deduce the main theorem. Section~\ref{examples} is devoted to examples and concluding remarks.


\section{Outline of the proof of Theorem~\ref{th-main1}} \label{outline}
Recall that for $f,g\in L^2(X_\zeta,\mu)$ the (complex) spectral measure $\sig_{f,g}$ is determined by the equations
$$\widehat{\sig}_{f,g}(-k) =\int_0^1 e^{2\pi i k \om}\,d\sig_{f,g}(\om)=
\langle f\circ T^k, g\rangle,\ \ k\in \Z,
$$
where $\langle \cdot,\cdot \rangle$ denotes the scalar product in $L^2$. We write $\sig_f = \sig_{f,f}$. Spectral measures ``live'' on the torus $\R/\Z$, which we identify with $[0,1)$.
It is known that $\sig_\One = \delta_0$, where $\One$ is the constant-1 function and $\delta_0$ is the Dirac point mass at $0$. Thus, the ``interesting part'' of the spectrum is generated by functions $f$ orthogonal to constants.
We   say that a function $f\in L^2(X_\zeta,\mu)$ is {\em cylindrical} if it depends only on $x_0$, the 0-th term of the sequence $x\in X_\zeta$. Cylindrical functions form an $m$-dimensional vector space, with a basis
$\{\One_{[a]}: \, a\in \Ak\}$. Denote
$$
\sig_\a:= \sig_{\One_{[\a]}}\ \ \ \mbox{and}\ \ \ \sig_{ab}:= \sig_{\One_{[a]},\One_{[b]}}.
$$
These measures are also known as {\em correlation measures}. 

The key fact which distinguishes constant-length substitutions is that the ``renormalization'', or ``substitution'' action on $X_\zeta$ is closely related  to the familiar times-$q$ map on the torus: $y\mapsto qy$ (mod 1) on $\R/\Z$, where $q$
is the length of the substitution. A measure on the torus is called $q$-mixing if it is invariant and mixing with respect to the times-$q$ map. By the Birkhoff Ergodic Theorem, any $q$-mixing measure is either singular or coincides with the Lebesgue (Haar) measure on the torus.
Further, let
\be \label{eq-convo}
\om = \sum_{n=1}^\infty 2^{-n}\om_{q^n},
\ee
where $\om_{q^n}$ is the Haar measure on the finite group generated by $1/q^n$ on the torus.
It is proved in \cite[Theorem 10.2 and Theorem 11.1]{Queff} that the maximal spectral type of $T_\zeta$ is given by
$$
\wtil{\lam} = \lam * \om,
$$
where $\lam = \lam_0 + \cdots +\lam_{k-1}$, with every $\lam_j$ being $q$-mixing, and moreover, $\lam_j$ is a linear  combination of correlation measures $\sig_{ab},\ a,b\in \Ak$. (In fact, $\lam_0 = \delta_0$.) It follows that in order to prove singularity we only need to rule out that any of the $\lam_j$, $j=1,\ldots,k-1$, is the Lebesgue measure.
We will show that every $\lam_j$, $j\ge 1$, can be expressed as a {\em positive} linear combination of spectral measures corresponding to cylindrical functions orthogonal to the constants. 
It turns out that the local dimension of the spectral measure $\sig_f$ for $f\perp \One$ at zero is equal to $d=2-2\alpha$, where $\alpha = \log|\theta_j|/\log q$ and $\theta_j$ is one of the non-maximal eigenvalues of $\Sf_\zeta$, with 
$|\theta_j|>1$, or else $\lim_{r\to 0} \frac{\sig_f(B(0,r))}{r^{2-\eps}}=0$ for any $\eps>0$. The latter is impossible for a measure with a Lebesgue component, and the former is ruled out by the 
assumptions of Theorem~\ref{th-main1} since we obtain $d\ne 1$.
This concludes the rough outline.


\subsection{Dimension of spectral measure at zero} 
Here we state the result about the dimension of spectral measure at zero, which is of some independent interest. It is a partial extension of a result from a joint work by A. I. Bufetov and the second author 
\cite{BuSo1}; its proof is also adapted from \cite[Section 6]{BuSo1}.
 
Recall that the local
dimension of a positive measure $\nu$ at a point $\om\in \R$ is defined by
$$
d(\nu,\om) = \lim_{r\to 0} \frac{\log \nu(B_r(\om))}{\log r},
$$
when the limit exists.
We order the eigenvalues of $\Sf_\zeta$ by magnitude: $\theta_1=q > |\theta_2| \ge\ldots$ For each eigenvalue $\theta_j$ denote by $P_j$ the ``natural'' projection onto the generalized eigenspace of the transpose $\Sf_\zeta^t$ corresponding to $\theta_j$, so that $P_j$ commutes with $\Sf^t_\zeta$. 

\begin{theorem} \label{th-dimzero}
Let $\zeta$ be a primitive substitution of constant length $q$, with a substitution matrix $\Sf_\zeta$ having eigenvalues $\theta_1=q, \theta_2, \ldots$, ordered by magnitude. Let $f$  be a cylindrical function on $X_\zeta$ given by $f = \sum_{k=1}^m b_k \One_{[k]}$ and consider $\vec{b} = (b_k)_{k=1}^m$. Let $j$ be minimal such that
$
P_j \vec{b} \ne 0.
$

{\bf (i)} 
Suppose that $|\theta_j|>1$. 
Then the spectral measure $\sigma_f$ satisfies
$$
d(\sig_f,0) = 2-2\alpha,\ \ \ \mbox{where}\ \ \alpha = \frac{\log |\theta_j|}{\log \theta_1}.
$$

{\bf (ii)}
If $|\theta_j|\le 1$, then 
$$
\lim_{r\to 0} \frac{\sig_f(B_r(0))}{r^{2-\eps}}= 0,\ \ \mbox{for any}\ \  \eps>0.
$$

\end{theorem} 

\begin{remark} {\em (1) We do not exclude the case $j=1$, which is equivalent to $P_1\vec{b} = \int_{X_\zeta} f \,d\mu\ne 0$, and then $\alpha =1$. In fact, then the spectral measure
$\sig_f$ has a point mass at zero and hence the local dimension at $\om=0$ equals zero.

(2) In \cite[Theorem 6.2]{BuSo1} a much stronger conclusion was obtained for the self-similar suspension flow over the substitution $\Z$-action, for a substitution of possibly non-constant length, but under the additional assumption that $\theta_2$ is positive and real of multiplicity one, strictly greater than $|\theta_3|$,
$f$ is orthogonal to constants, and  $j=2$. Then it was shown, in particular, that $\sig_f(B_r(0)) \asymp r^{2-2\alpha}$, with $\alpha = \log\theta_2/\log\theta_1$. We partially extend this result, removing the eigenvalue assumption, and use the fact that in the constant-length case the self-similar suspension flow is obtained simply by taking the constant-one roof function.}
\end{remark}


\section{Results of Queff\'elec and their consequences}\label{Queffelec}

Let $\zeta$ be a substitution of constant length $q$, primitive, aperiodic, and of height one. We summarize the results of Queff\'elec \cite{Queff} on the maximal spectral type. For a reader unfamiliar with the procedure below, it may be useful to follow along with Example~\ref{ex1}.

\begin{defi} {\em (See \cite[10.1.1]{Queff}.)}
Consider the bi-substitution (or ``square'' substitution) $\zeta^{[2]}$ defined on the alphabet $\Ak\times\Ak$ as follows:
$$
\zeta(a,b) = (\zeta(a)_1,\zeta(b)_1)\ldots(\zeta(a)_q,\zeta(b)_q).
$$
It is clearly also of constant length $q$. Its substitution matrix is denoted by ${\sf C}$ and is called the {\bf coincidence matrix} for $\zeta$.
\end{defi}

We continue with the definitions, following \cite[10.1.1]{Queff}. For $(a,b)\in \Ak\times \Ak$ let $\Ok(a,b)$ be the set of 
all pairs $(c,d)\in \Ak\times\Ak$ appearing in 
${(\zeta^{[2]})}^n(a,b)$ for some $n$.
Minimal sets of the form $\Ok(a,b)$ with respect to inclusion are disjoint; they are called {\em ergodic classes} of the bi-substitution and denoted $E_0,\ldots, E_{k-1}$. One always has $1\le k \le m$. The class of ``coincidence pairs'' $\{(a,a):\ a\in \Ak\}$ always forms an ergodic class (this follows from primitivity), denoted by $E_0$. We note that $\zeta^{[2]}$ is never primitive. The class of remaining pairs, if any, is called {\em transitive} and denoted $T$; for any $(a,b)\in T$, the orbit $\Ok(a,b)$ must intersect one of the ergodic classes. Passing from $\zeta$ to $\zeta^j$, if necessary, one can always assume that ${\sf C}$ restricted to any ergodic class is primitive.
It is not hard to see that the multiplicity of the Perron-Frobenius (PF) eigenvalue $q$ of $\Cf$ is exactly $k$, the number of ergodic classes, \cite[Lemma 10.1]{Queff}. The next theorem is \cite[Prop.\ 10.1]{Queff}; it is essentially a reformulation of Dekking's criterion \cite{Dek1} mentioned above.

Denote the entries of $\Cf$ by $\Cf_{ab}^{cd}$, where $(a,b), (c,d)\in \Ak\times \Ak$. Consider also the vector of dimension $m^2$ of all correlation measures
$$
\Sigma:= (\sig_{ab})_{(a,b)\in \Ak\times \Ak}.
$$
Thus $\Cf\,\Sigma$ is a well-defined vector of measures. Queff\'elec \cite[p.\,244]{Queff} proves an interesting formula, which already hints at times-$q$ invariance of spectral measures:
$$
S_q(\Sigma) = \frac{1}{q} \Cf\, \Sigma,
$$
where $S_q(\Sigma)$ denotes the pull-back of $\Sigma$ under the times-$q$ map on $\R/\Z$.

\begin{theorem}{\em (Dekking)} 
Let $\zeta$ be a primitive aperiodic substitution of constant length $q$ and height one. Then $(X_\zeta,T_\zeta,\mu)$ has purely discrete spectrum if and only if $E_0$ is the only ergodic class of $\Cf$.
\end{theorem}

From now on, we assume that the substitution is not purely discrete and there are at least two ergodic classes of $\Cf$. Denote by $F$ the $k$-dimensional eigenspace of $\Cf^t$ corresponding to the PF eigenvalue $q$. Since $\zeta$ has constant length $q$, it follows that every vector in $F$ is {\em constant} on each ergodic class and zero on the transitive class,
see \cite[Prop.\ 10.2]{Queff} for details. (The basic reason is that $(1,\ldots,1)$ is the PF left eigenvector for any constant length substitution.)

\begin{defi} {\em (See \cite[p.251]{Queff})} To every vector $v\in \C^{\Ak\times \Ak} = \C^{m^2}$ we associate the $m\times m$ matrix $(v_{ab})_{a,b\in \Ak}$. The vector 
$v$ is called {\bf strongly positive}, denoted $v\gg 0$, if the associated matrix is Hermitian positive semi-definite. Denote by $F_+$ the set $\{v\in F:\ v\gg 0\}$.
\end{defi}

For $v=(v_{ab})_{a,b\in \Ak}$ let
$$
v\Sigma = \sum_{a,b\in \Ak} v_{ab}\sig_{ab}.
$$

Now we can state the theorem of Queff\'elec, which identifies the maximal spectral type. It is a combination of \cite[Theorem 10.1, Theorem 10.2, and Theorem 11.1]{Queff}.

\begin{theorem} {\em (Queff\'elec)} \label{th-queff} Let $\zeta$ be a primitive aperiodic substitution of constant length $q$ and height one.
Consider
$$
\Qk:= \{v\in F_+:\ v_{aa}=1,\ a\in \Ak\},
$$
using the notation introduced above,
and let $v^{(0)},\ldots, v^{(k-1)}$ be the extreme points of the convex set $\Qk$. Then $\lam_i = v^{(i)}\Sigma$, $i=0,\ldots,k-1$, are distinct $q$-mixing probability measures, with $v^{(0)} = (1,\ldots,1)$ and $\lam_0 = \delta_0$, and the maximal spectral type of $(X_\zeta,T_\zeta,\mu)$ is
$$
(\lam_0 + \cdots + \lam_{k-1}) * \om,
$$
where $\om$ is defined by (\ref{eq-convo}).
\end{theorem}

Note that it follows from the Birkhoff Ergodic Theorem that distinct $\lam_i$ are mutually singular.

\begin{prop}\label{decomposition}
Every measure $\lam_i$ for $i=1,\ldots,k-1$, can be represented as a sum of spectral measures $\sig_f$, where $f$ is a cylindrical function on $X_\zeta$ orthogonal to constants.
\end{prop}

\begin{proof}
We have $\lam_i = v^{(i)}\Sigma$ by Queff\'elec's Theorem. Let $(v^{(i)}_{ab})_{a,b\in \Ak}$ be the positive semi-definite $m\times m$ matrix associated to $v^{(i)} \in \C^{\Ak\times \Ak}$. It can be diagonalized: there exists an  orthonormal basis $\{(d_{j,a})_{a\in \Ak},\ j=1,\ldots,m\}$, such that 
$$
v^{(i)}_{ab} = \sum_{j=1}^m \kappa_j \ov{d_{j,a}} d_{jb}\ \ \Longrightarrow\ \ \lam_i = \sum_{a,b} v^{(i)}_{ab} \sig_{ab} = \sum_{j=1}^m \kappa_j \Bigl(\sum_{a,b}  \ov{d_{j,a}} d_{j,b}\sig_{ab}\Bigr),
$$  
for some $\kappa_j\ge 0$. We then let $f_j = \sqrt{\kappa_j} \sum_{a\in \Ak} \ov{d_{j,a}}\One_a$ and observe that $\lam_i = \sum_{j=1}^m  \sig_{f_j}$. (The latter follows from the basic properties of spectral measures: $\sig_{f+g}= \sig_f + \sig_g + \sig_{f,g} + \sig_{g,f},\ \sig_{f,g} = \ov{\sig_{g,f}},\ \sig_{cf,g} = c\,\sig_{f,g}$.) 
It remains to show that all $f_j$, for which $\kappa_j>0$, are
orthogonal to constants. We have $\lam_0 = \delta_0$ and $\lam_i\perp \lam_0$ for $i\ge 1$, hence $\lam_i(\{0\})=0$. Since $\sig_f(\{0\})=0$ if and only if $\langle f,\One\rangle_{L^2}=0$, the claim follows.
\end{proof}

\begin{corollary} \label{cor-leb} Under the assumptions of Theorem~\ref{th-queff}, if $(X_\zeta,T_\zeta,\mu)$ has a Lebesgue component in the spectrum, then for some cylindrical functions $f_j$ we have
$\sum_j \sig_{f_j}([0,r]) = r$.
\end{corollary}

At this point we can already deduce ``one half'' of the main Theorem~\ref{th-main1}, modulo the reduction to pure base in the next section.

\begin{corollary} \label{cor-half} Let $\zeta$ be a primitive aperiodic substitution of constant length $q$. If the substitution matrix $\Sf_\zeta$ has the second in modulus eigenvalue $\theta_2$, satisfying
$|\theta_2| < \sqrt{q}$, then the maximal spectral type of the substitution dynamical system is singular.
\end{corollary}

\begin{proof}
As we show in the next section (see Proposition~\ref{prop-reduce} below), without loss of generality we can assume that the height of substitution equals 1.
In \cite[Corollary 3.11]{BuSo1} it is proved that for any cylindrical function $f$ orthogonal to constant, under the conditions of the corollary, for any $\eps>0$:
\be \label{Adam}
\sig_f(B_r(0))  = o(r^{\alpha-\eps}),\ \ \mbox{as}\ r\to 0,\ \ \mbox{where}\ \ \alpha = \log_q|\theta_2|< 1/2,
\ee
and this  implies the claim, in view of Corollary~\ref{cor-leb}. (We should note that (\ref{Adam})  follows from the work of Adamczewski \cite{Adam}, with an additional argument given in \cite{BuSo1}.)
\end{proof}

The remaining case of the theorem, when $|\theta_2|>\sqrt{q}$ and no other eigenvalue of $\Sf_\zeta$ has absolute value $\sqrt{q}$, will require extra work, done in Section 5.

\section{Reduction to the pure base.}\label{purebase}

In this section we prove that the difference in the spectrum of the substitution
matrix and the pure base substitution matrix can only be in 0 or in the roots of unity,
and therefore our main result holds for substitutions of arbitrary height. 

For the proof we rely on the results on Durand (\cite{Du}) concerning the spectrum of return word substitutions.
For a primitive substitution $\zeta$ with the fixed point $U\in{\mathcal A}^\mathbb N$ and its prefix 
$u\in{\mathcal A}^+$, we call a subsequence $v$ of $U$ {\it a return word} on $u$ if $u$ is
a prefix of $vu$ and the word $u$ appears in the sequence $vu$ exactly twice.

Since the substitution $\zeta$ is primitive, every block occurs in $U$ with bounded gaps, hence the set of return words, denoted ${\mathcal R}_{U,u}$, is finite.
 There is a unique way to write the sequence $U$ as a concatenation
of return words. The substitution $\zeta$ induces a substitution on the set of return words, which we denote
by $\Theta_{U,u}\colon {\mathcal R}_{U,u}\to {\mathcal R}_{U,u}^+$. By \cite[Proposition 9]{Du} the matrices of substitutions
$\zeta$ and $\Theta_{U,u}$ have the same spectrum, with the possible exceptions of 0 and the roots of unity.

\begin{prop} \label{prop-reduce}
Let $\zeta$ be a substitution of height $h> 1$ with substitution matrix $S_\zeta$, and $\eta$ its corresponding
pure base substitution (cf. \cite[Lemma 17]{Dek1}) with substitution matrix $S_\eta$. Then  the spectra of matrices $S_\eta$ and $S_\zeta$ may differ only by 0 or the roots of unity.
\end{prop}
\begin{proof}
We prove that $\zeta$ and $\eta$ have prefixes, such that the induced substitutions
on the sets of return words coincide up to a permutation of the alphabets. Denote by $U = u_0u_1u_2\ldots$ the fixed point of substitution $\zeta$. 
By \cite[Theorem 14, Lemma 17]{Dek1}, the alphabet $\mathcal I$ for the pure base
substitution $\eta$ consists of the different blocks of symbols from $U$ of length $h$, that start
at places $hk$, $k\in\mathbb N\cup\{0\}$. Therefore there exists a 1-to-1 map $\phi\colon{\mathcal I}
\to{\mathcal A^h}$ extendable by concatenation to an arbitrary sequence of letters from ${\mathcal I}$
such, that if $V=v_0 v_1 v_2 \ldots$ is the fixed point for substitution $\eta$,
then $U=\phi(v_0)\phi(v_1)\phi(v_2)\dots$ and $\phi\circ\eta=\zeta\circ\phi$.

Consider the set of return words ${\mathcal R}_{V,v_0}$ with the induced substitution 
$\Theta_{V,v_0}\colon {\mathcal R}_{V,v_0}\to {\mathcal R}_{V,v_0}^+$ and the set of return words
${\mathcal R}_{U,\phi(v_0)}$ with the induced substitution 
$\Theta_{U,\phi(v_0)}\colon {\mathcal R}_{U,\phi(v_0)}\to {\mathcal R}_{U,\phi(v_0)}^+$. Note that there
exists a 1-to-1 correspondence $\psi\colon {\mathcal R}_{V,v_0}\to {\mathcal R}_{U,\phi(v_0)}$, 
such that for all $w\in {\mathcal R}_{V,v_0}$ we have $\psi(w)=\phi(w)$. Suppose that the return word $r$
is located in the sequence $V$ between positions $i$ and $j$, i.e.
$r=v_i v_{i+1}\ldots v_{j-1}$, $v_i=v_0$, $v_j=v_0$, and $v_l\ne v_0$ for all $i<l<j$.
In the sequence $\phi(r)=\phi(v_i)\phi(v_{i+1})\ldots\phi(v_{j-1})$ the word $\phi(v_0)$
appears only once, because the word $\phi(v_0)$ starts with symbol $u_0$, which can only be located
in places numbered $hk$, $k\in{\mathbb N}\cup \{0\}$ (see the definition of height and \eqref{height}),
and therefore its appearance for the second time must coincide with one of $\phi(v_l)$, $i<l<j$, which
is impossible, hence $\phi(r)$ is a return word in ${\mathcal R}_{U,\phi(v_0)}$.

Since $\Theta_{V,v_0}$ is induced by $\eta$, if $r\in {\mathcal R}_{V,v_0}$ we have 
$\Theta_{V,v_0}(r)=r_1r_2\ldots=\eta(r)$ for some $r_1,r_2,\ldots\in {\mathcal R}_{V,v_0}$,
and $\phi\circ\eta(r)=\zeta\circ\phi(r)$. Hence $\Theta_{U,\phi(v_0)}(\psi(r))=\zeta(\phi(r))=\phi(\eta(r))=\phi(r_1)\phi(r_2)\ldots=\psi(r_1)\psi(r_2)\ldots$, therefore
$\psi$ is a bijection and the substitutions on return words coincide up to a renumbering of the alphabet.
\end{proof}
In example~\ref{addeigen} in Section~\ref{examples} we show that a change can occur in the eigenvalues
of substitution matrix according to our last proposition.

\medskip

In order to complete the reduction to the pure base in our Theorem~\ref{th-main1}, it remains to observe that by \cite{Dek1} (see also \cite[6.3.1.2]{Queff}), the initial system $(X_\zeta,T,\mu)$ is metrically isomorphic to the tower  of constant height $h$ over the pure base substitution system $(X_\eta,T',\mu')$, hence the maximal spectral type of the former is obtained by the convolution of the spectral type of the latter with the Haar measure on $\Z/h\Z$.


\section{Dimension of spectral measure at zero}\label{dimension}
In order to compute the dimension of spectral measure at zero, we extend some of the results from the joint work of A. I. Bufetov and the second author \cite[Section 6]{BuSo1} developed for self-similar suspension  flows over substitution systems (not necessarily of constant length). We recall this more general setting.

\subsection{General case}
Suppose that $\zeta$ is a primitive aperiodic substitution on an alphabet $\A$ of size $m\ge 2$, with a substitution matrix $\Sf=\Sf_\zeta$. We emphasize that in this subsection we consider  both constant-length and non-constant length substitutions.
Let
\[
\theta_1> |\theta_2|\ge\cdots \ge|\theta_m|
\]
be the eigenvalues of $\Sf$, counted with algebraic multiplicities,
in the order of descent by their absolute value.
Let $s$ be the Perron-Frobenius (PF) eigenvector for the transpose $\Sf^t$, normalized by
$$
\sum_{a\in \A} s_a \mu([a])=1.
$$

The self-similar suspension flow over the uniquely ergodic substitution measure-preserving $\Z$-action $(X_\zeta,T_\zeta,\mu)$ is defined as follows:
$$\Xxi = \{x=(y,t):\ y\in X_\zeta,\ 0 \le t \le s_{y_0}\}/_\sim\ ,$$
 where  the relation $\sim$ identifies the points $(y, s_{y_0})$ and $(T_\zeta (y), 0)$. The measure $\wt{\mu}$ on $\Xxi$ is induced from the product  of $\mu$ on the cylinder sets $[a]\subset X_\zeta$ and the Lebesgue measure; it is invariant under the action $h_\tau(y,t) = (y, t+\tau)$, which is defined for all $\tau\in \R$ using the identification $\sim$. Our normalization ensures that $\wt{\mu}$ is a probability measure, and
we thus obtain a measure-preserving $\R$-action $(\Xxi, h_\tau,\wt{\mu})$. The term ``self-similar'' comes from another representation of this flow, as a tiling dynamical system on the line, with $(y,\tau)$ corresponding to the tiling 
of $\R$ by closed interval tiles of lengths $u_{1,y_n}$, $n\in \Z$, in such a way that the origin  is located at the point $t$ in the tile labelled by $y_0$, see \cite{SolTil,BuSo1} for details.
The substitution action $Z$ on $\Xxi$ is then defined by: scale the entire tiling by $\theta_1$ and subdivide the resulting tiles according to the substitution rule. These actions satisfy the relation
$$
Z\circ h_\tau = h_{\theta_1 \tau} \circ Z.
$$

Denote by $P_j$ the projection of $\C^m$ onto the span of generalized eigenvectors of $\Sf^t$, corresponding to $\theta_j$, commuting with $\Sf^t$.
The following notation will be useful: for $v\in \C^m$ let
\be \label{def-j}
j(v) = \min\{i:\ P_i(v)\ne 0\},
\ee
\be \label{def-pr}
pr(v) = \sum_{k: |\theta_k|=|\theta_{j(v)}|} P_k v, 
\ee
and let $\kappa(v)$ be the ``largest elimination exponent'' of $v$, defined by
\be \label{def-kappa}
\kappa(v) = \max\{i:\ (\Sf^t - \theta_k I)^{i-1} P_k(v)\ne 0:\ |\theta_k|=|\theta_{j(v)}|\}.
\ee

Denote by $E^+$ be the linear span of generalized eigenvectors for $\Sf^t$ corresponding to eigenvalues greater than one in absolute value.
Let $\{\Phi^+_{v,x}:\ v\in E^+, \ x\in \Xxi\}$ be the H\"{o}lder cocycle (family of finitely-additive measures), 
defined on the algebra generated by line segments in $\R$
(see \cite{B13,BS13,BuSo1}). We refer to these papers for the definition of this cocycle, but collect their  properties that we need in the lemma below. For simplicity, we write $\Phi^+_{v,x}(t) = \Phi^+_{v,x}([0,t])$.

\begin{lemma} \label{lem-cocycle}
\renewcommand{\theenumi}{\roman{enumi}}
Under the standing assumptions, we have:
\begin{enumerate}
\item Linearity: $v\mapsto \Phi^+_{v,x}(t)$ is a linear function on $E^+$;

\item Cocycle property: $\Phi^+_{v,x} (s+t) = \Phi^+_{v,x} (s) + \Phi^+_{v,h_s(x)} (t)$ for all $v\in E^+$, $x\in \Xxi$, $s, t \in \R$;

\item Renormalization: $\Phi^+_{v,Z(x)} (\theta t) = \Phi^+_{\Sf^t v,x}(t)$;

\item Upper bound: for $v\in E^+$ we have
\[
|\Phi^+_{v,x}(t)|\le C_1 \|v\|\cdot |t|^\alpha (\log|t|)^{\kappa(v)-1},\ \ \mbox{for}\ t\ge 2,\ \ \mbox{where}\ \ \alpha = \frac{\log|\theta_{j(v)}|}{\log\theta}\,,
\]
and $C_1$ depends only on the substitution $\zeta$.

\item Continuity: for a given $v\in E^+\setminus \{0\}$, the function $(t,x)\mapsto \Phi^+_{v,x}(t)$ is continuous on $\R_+\times \Xxi$;

\item Non-degeneracy: for a given $x\in \Xxi$ and $v\in E^+\setminus \{0\}$, the function $t\mapsto \Phi^+_{v,x}(t)$ is  not constant zero on any interval.
\end{enumerate}
\end{lemma}

\begin{proof}
Although all these statements are essentially contained in \cite{B13}, the setting there is somewhat different (Vershik automorphisms on Bratteli diagrams), so we prefer to give a more direct derivation based on \cite{BuSo1}. First we note that in \cite{BuSo1} the finitely-additive measures are defined on a class of subsets  of $\R^d$, they correspond to self-similar tilings of $\R^d$ and vectors $v$ in a subspace $E^{++}\subset E^+$; however, in the case $d=1$ needed here we have
$E^{++}=E^+$. 

Linearity (i) is immediate from the definition \cite[Eq.(9)]{BS13} and (ii) is proved in \cite[Lemma 3.2]{BS13}. The cocycle property is verified just before Lemma 6.1 in \cite{BuSo1}. The estimate (iv) is contained in \cite[Lemma 3.3]{BS13}.

(v) In \cite[Lemma 3.5]{BS13} it was assumed that $v$ is an eigenvector for $\Sf^t$, but it easily extends to the general case of $v\in E^+$ (for $d=1$) with $|\theta|$ in that lemma replaced by any $|\theta|-\eps>1$,
with $|\theta|$ the minimal absolute value of an eigenvalue greater than 1, and allowing the constant $C$ in \cite[Eq.(29)]{BS13} to depend on $r_1< r_2$ as well. 
Continuity of $t\mapsto \Phi^+_{v,x}(t)$  then immediately follows from the modified H\"older estimate \cite[Eq.(26)]{BS13}. Continuity in the $x$ variable along the orbit follows from (ii). On the ``transversal'' consisting of all tilings which agree with a given $x$ on a tile containing the origin, $x'$ being ``close'' in the tiling metric means perfect agreement on a neighborhood $B_R(0)$ for a large $R$, in which case $\Phi^+_{v,x}(t) = \Phi^+_{v',x'}$ for all
$t\in (0,R)$.

(vi) This follows from (v) and from the fact that $\Phi^+_{v,x}([t_1,t_2])$ is non-zero on the line segments that are ``tiles'' of $x$, as well as on their ``subtiles'' obtained by sibdividing the tiles according to the substitution rule as many times as we wish; thus for a dense set of pairs $t_1< t_2$.
\end{proof}

Estimates of the dimension of spectral measures at zero will be derived from the asymptotics of ergodic (Birkhoff) integrals, which we state next.
Denote 
\be
S(F,x, t) = \int_0^t F\circ h_\tau(x)\,d\tau,\ \ \mbox{for}\ \  F\in \Xxi.
\ee
As test functions, we consider cylindrical functions on $\Xxi$ defined by $F=\sum\limits_{a\in \A} F_a \One_{\wt a}$, where $\One_{\wt a}=\One_{[{\wt a}]}$ and $[{\wt a}] = \{(y,t)\in \Xxi:\ y_0 = a\}$ is the cylinder set for the suspension flow (note that this definition is slightly less general than in \cite{BuSo1}; we restrict ourselves to such functions for simplicity of exposition, since this is all that we need for our application).
A cylindrical function $F$ is uniquely determined by the vector $\vec{F}:=(F_a)_{a\in \A} \in \C^m$. 
The following estimate is a special case of \cite[Theorem 4.3]{BS13}. 

\begin{corollary}[\cite{BS13}] \label{cor-BS} Let $F=\sum\limits_{a\in \A} F_a \One_{\wt a}$, with $\|\vec{F}\|=1$, and let $j=j(\vec{F})\ge 1$ be minimal such that $P_j \vec{F}\ne 0$. 
\renewcommand{\theenumi}{\roman{enumi}}
\begin{enumerate}
\item Suppose that $|\theta_j|>1$. Then for any $\eps>0$ there exists $C_\eps>0$ such that
\be \label{est1}
S(F,x,t) = \Phi^+_{pr(\vec{F}),x}(t)+ \Rk(t),\ \ \ |\Rk(t) | \le C_\eps \max\{1, |t|^{\alpha-\eps}\},
\ee
where 
$pr(\vec{F})$ is defined in (\ref{def-pr}) and $\alpha = \frac{\log|\theta_j|}{\log|\theta_1|}$.

\item If $|\theta_j|\le 1$, then
\be \label{est2}
|S(F,x,t)| \le C \max\{1, \bigl|\log |t|\bigr|^k\},
\ee
for some $k\ge 0$.
\end{enumerate}
\end{corollary}

\begin{proof}
Choose a basis $\{u_k\}_{k=1}^m$ for $\C^m$ consisting of generalized eigenvectors of $\Sf$ and the dual basis $\{v_k\}_{k=1}^m$, so that $u_k$ corresponds to the eigenvalue $\theta_k$.
Applying \cite[Theorem 4.3]{BS13} in the case $d=1$, with $\Omega=[0,1]$, yields
$$
S(F,x,t) = \sum_{k:\ |\theta_k|>1} \Phi_{v_k,x}^+(t) \cdot \langle \vec{F},u_k\rangle +  C \max\{1, \bigl|\log |t|\bigr|^k\}.
$$
If $|\theta_j|\le 1$, then  $\langle \vec{F},u_k\rangle =0$ for  all $k$ such that $|\theta_k|>1$, and (\ref{est2}) follows. If $|\theta_j|>1$, then, 
taking \cite[Lemma 3.3]{BS13} into account and the fact that $\langle \vec{F},u_k\rangle =0$ for $|\theta_k|>|\theta_j|$ by the definition of $j$, we obtain
$$
S(F,x,t) = \sum_{k:\ |\theta_k|= |\theta_j|} \Phi_{v_k,x}^+(t) \cdot \langle \vec{F},u_k\rangle + O(|t|^{\alpha-\eps}),\ \ |t|\to\infty.
$$
It remains to note that 
$$
pr(\vec{F}) = \sum_{k:\ |\theta_k|=|\theta_j|} \langle \vec{F},u_k\rangle v_k,
$$
so the desired estimate (\ref{est1}) follows from the linearity of $v\mapsto \Phi_{v,x}^+(t)$.
\end{proof}

\begin{prop}\label{prop-dimzero}
Consider a cylindrical function $F=\sum\limits_{\alpha\in \A}F_\alpha \1_\alpha$, with $\|\vec{F}\|=1$ and let $\sig_F$ be the spectral measure for the measure-preserving flow $(\Xxi,h_\tau,\wt{\mu})$ corresponding to $F$. Let $j=j(\vec{F})$.
\renewcommand{\theenumi}{\roman{enumi}}
\begin{enumerate}
\item Suppose that
the eigenvalue $\theta=\theta_j$ satisfies
$1< |\theta| < \theta_1$. Let $\kappa=\kappa(\vec{F})$ be the ``largest elimination exponent'' defined in (\ref{def-kappa}). Then
\be \label{bounds}
0 < \liminf_{r\to 0}\frac{\sigma_F(B_r(0))}{|\log r|^{\kappa-1} \cdot r^{2-2\alpha}}\le  \limsup_{r\to 0}\frac{\sigma_F(B_r(0))}{|\log r|^{\kappa-1} \cdot r^{2-2\alpha}} < \infty,
\ee
where $\alpha=\log|\theta|/\log \theta_1$; thus $d(\sig_F,0) = 2-2\alpha$.

\item If $|\theta_j|\le 1$, then for some $k\ge 0$ holds
$$
\limsup_{r\to 0}\frac{\sigma_F(B_r(0))}{|\log r|^{k} \cdot r^{2}} < \infty.
$$
\end{enumerate}
\end{prop}
\begin{proof} (ii) We start with the second part, which is the easy case. In fact, this is just a consequence of Corollary~\ref{cor-BS}(ii), combined with \cite[Lemma 4.3]{BuSo1}.

(ii)
Recall that the spectral measure $\sigma_F$ for the $\R$-action satisfies
\[
\int_\R e^{2\pi i \omega t}d\sigma_F(\omega)=\langle F\circ h_t,F\rangle.
\]
Fix a Schwartz function $\psi$, and write, using 
the Inverse Fourier Transform,
\[
\psi(T\omega)=T^{-1}\int_\R e^{2\pi i \omega t}\what\psi(t/T)\,dt,\ \ T>0.
\]
In view of the spectral isomorphism between $L^2({\mathbb R},\sigma_F)$ and a closed subspace
of $L^2(\Xxi,\wt{\mu})$, that carries multiplication by
$e^{2\pi i \omega t}$ into composition with $h_t$, we obtain
\be \label{eqa0}
\int_\R |\psi(T\omega)|^2\,d\sigma_F(\omega)=T^{-2}\int_{\Xxi}
\bigg\vert \int_\R F\circ h_t(x)\cdot \what{\psi}(t/T)\,dt\bigg\vert^2d\wt{\mu}(x).
\ee
Integrating by parts and using Corollary~\ref{cor-BS}, we obtain
\begin{eqnarray}
\int_\R F\circ h_t(x)\cdot \what{\psi}(t/T)\,dt & = &
-T^{-1}\int_\R S(F,x,t)\cdot(\what{\psi})'(t/T)\,dt \nonumber \\[1.2ex]
 & = & -T^{-1}\int_\R\bigl(\Phi^+_{pr(\vec{F}),x}(t)+{\mathcal R}(t)\bigr)\cdot(\what{\psi})'(t/T)\,dt. \label{eqa1}
\end{eqnarray}
For a Schwartz function $\what\psi$, its derivative can be estimated 
$
|(\what{\psi})'(t)|\le C_{\psi,\alpha}\min(1,|t|^{-\alpha-1}),
$
hence (\ref{est1}) yields
\begin{eqnarray}
\bigg\vert\int_\R {\mathcal R}(t)\cdot (\what{\psi})'(t/T)\,dt\bigg\vert & \le & \nonumber \\
& \le  & 2C_\eps C_{\psi,\alpha}\left(1+\int_1^Tt^{\alpha-\epsilon}\,dt
+\int_T^\infty t^{\alpha-\epsilon}(t/T)^{-\alpha-1}\,dt\right)
=O(T^{\alpha+1-\epsilon}), \label{errest}
\end{eqnarray}
as $T\to\infty$. In order to analyze the main term in (\ref{eqa1}), we make a change of variable $t=T\tau$, assuming $T=\theta_1^N$ for $N\ge 1$, to obtain
$$
-T^{-1}\int_\R \Phi^+_{pr(\vec{F}),x}(t)\cdot(\what{\psi})'(t/T)\,dt = - \int_\R \Phi^+_{pr(\vec{F}),x}(\theta_1^N \tau)\cdot(\what{\psi})'(\tau)\,d\tau.
$$
We use the Jordan decomposition 
of the matrix $\Sf^t$. Denote the Jordan blocks with eigenvalues equal by absolute value
to $|\theta|$ by $J_{k_1},\ldots,J_{k_l}$, with the corresponding invariant subspaces
$V_{k_1},\ldots,V_{k_l}$ and eigenvalues $\theta_{k_1},\ldots,\theta_{k_l}$
such that $|\theta_{k_i}|=|\theta|, \ 1\le i \le \ell$. (Note that there may be several Jordan blocks corresponding to the same eigenvalue, so $\theta_{k_i}$ need not be distinct.)
We then have
$$
pr(\vec{F}) = \sum_{i=1}^\ell pr_i(\vec{F}),\ \  \ \mbox{hence}\ \ \ \Phi^+_{pr(\vec{F}),x}(\theta_1^N \tau) = \sum_{i=1}^\ell \Phi^+_{pr_i(\vec{F}),x} (\theta_1^N \tau),
$$
where $pr_i$ is the projection of $\C^m$ onto the subspace $V_{k_i}$, parallel to the other invariant subspaces of $\Sf^t$.  Fix $i$ and denote by $\{w_j\}_{j=1}^{p_i}$ the basis
of $V_{k_i}$ consisting of eigenvector and root vectors, so that $\Sf^t w_{p_i} = \theta_{k_i} w_{p_i}$ and $\Sf^t w_j = \theta_{k_i} w_j +  w_{j+1}$ for $j=1,\ldots, p_i-1$. 
By Lemma~\ref{lem-cocycle}(iii), using the Jordan form structure, we obtain for $j=1,\ldots,p_i$,
\be \label{renorm1}
\Phi^+_{w_j,x}(\theta_1^N \tau) = \Phi^+_{(\Sf^t)^N w_j, Z^{-n}(x)}(\tau) = \sum_{k=j}^{p_i} {N\choose j-k} \theta_{k_i}^{N-k+j} \Phi^+_{w_k, Z^{-n}(x)}(\tau).
\ee
Let $\kappa_i=\kappa(pr_i (\vec{F}))$ be the corresponding maximal elimination exponent, and denote by $pr'_i(\vec{F})$ the projection of $\vec{F}$ onto the linear span of $w_{\kappa_i}$. In other words, for some
$a_{\kappa_i},\ldots,a_{p_i}$,
$$
pr'_i(\vec{F}) = a_{\kappa_i} w_{\kappa_i},\ \ \ \mbox{whereas}\ \ \ pr_i(\vec{F}) = \sum_{j=\kappa_i}^{p_i} a_j w_j,\ \ a_{\kappa_i}\ne 0.
$$
Expand $\Phi^+_{pr_i(\vec{F}),x} (\theta_1^N \tau)$ by linearity, and apply (\ref{renorm1}) to each term. 

\begin{lemma} \label{lem-vspom1}
We have
\begin{equation*}
\lim\limits_{N\to\infty}\int_{{\mathcal X}_\zeta}\int_{\mathbb R}\bigg(\frac{
\Phi_{pr_{i}(\vec{F}),x}^+(\theta_1^N\tau)}
{\binom{N}{N-\kappa_i+1}\theta_{k_i}^{N-\kappa_i+1}}-
\Phi_{pr'_{i}(\vec{F}),x}^+(\tau)\bigg) (\what{\psi})'(\tau)\,d\tau\,
d\wt{\mu}(x)
=0.
\end{equation*}
\end{lemma} 

\begin{proof}
It follows from (\ref{renorm1}) that ${\binom{N}{N-\kappa_i+1}\theta_{k_i}^{N-\kappa_i+1}}\Phi_{pr'_{i}(\vec{F}),x}^+(\tau)$ is the term with the fastest growing coefficient in the decomposition of $\Phi^+_{pr_i(\vec{F}),x} (\theta_1^N \tau)$. The other terms, when divided by this  coefficient, tend to zero by Lemma~\ref{lem-cocycle}(iv), and the claim follows by Dominated Convergence.
\end{proof}

Now we can finish the proof of the proposition. Let $c_i = \theta/\theta_{k_i}$ (note that $|c_i|=1$). Recall that $\kappa = \kappa(\vec{F})= \max_i \kappa_i(\vec{F})$. Since $\theta_{k_i} = \theta/c_i$, the last lemma, together with (\ref{eqa1}) and (\ref{errest}),  implies
\begin{eqnarray} 
& & \liminf\limits_{N\to\infty} \frac{\textstyle \theta_1^{-2N}\int\limits_{\Xxi}
\bigg\vert \int\limits_{\R} F\circ h_t(x)\cdot \what{\psi}(t/\theta_1^N)\,
dt\bigg\vert^2d\wt{\mu}(x)}{\textstyle \theta_1^{(N-\kappa+1)(2\alpha-2)}\binom{N}{N-\kappa+1}} =  \nonumber
\\[1.2ex] \label{lim1}
& & \liminf\limits_{N\to\infty} \int\limits_{\Xxi}\bigg\vert
\sum\limits_{i:\, \kappa_i=\kappa} c_i^{N-\kappa+1}
 \int\limits_{\R}\Phi^+_{pr'_i(\vec{F}),x}(\tau)(\what{\psi})'(\tau)d\tau\bigg\vert^2 d\wt{\mu}(x).
\end{eqnarray}
We next use a trivial estimate
$$
\Bigl|\sum_{i:\,\kappa_i=\kappa}  c_i^{N-\kappa+1} B_i \Bigr| \ge\Bigl| \sum_{i:\,\kappa_i=\kappa} B_i \Bigr| - \max_i|B_i|\cdot \sum_{i:\,\kappa_i=\kappa}|1-c_i^{N-\kappa+1}|,
$$
where $B_i = \int\limits_{\R}\Phi^+_{pr'_i(\vec{F}),x}(\tau)(\what{\psi})'(\tau)d\tau$, and note that $\max_i |B_i| \le C_\psi$ by Lemma~\ref{lem-cocycle}(iv).

\begin{lemma} \label{lem-vspom2} For any complex numbers $c_i$, with $|c_i|=1$, $i=1,\ldots,\ell$, and any $\delta>0$, there exists a sequence $n_j \uparrow \infty$ such that 
$\sup_j (n_{j+1}-n_j) < \infty$ and
$$
\sum_{i=1}^\ell |1-c_i^{n_j}| < \delta.
$$
\end{lemma}

The proof of the lemma is elementary and is left to the reader. Using $N_j-\kappa+1=n_j$ from the lemma, we can estimate
$$
\Bigl|\sum_{i:\,\kappa_i=\kappa}  c_i^{N_j-\kappa+1} B_i \Bigr|^2 \ge  \Bigl| \sum_{i:\,\kappa_i=\kappa} B_i \Bigr|^2 - 2 \Bigl| \sum_{i:\,\kappa_i=\kappa} B_i \Bigr| C_\psi \delta \ge
\Bigl| \sum_{i:\,\kappa_i=\kappa} B_i \Bigr|^2 - 2\ell C_\psi^2 \delta.
$$
Thus, we obtain, denoting $pr'(\vec{F}) = \sum\limits_{i:\,\kappa_i=\kappa} pr'_i(\vec{F})$,
\begin{multline*}
\liminf\limits_{j\to\infty} \frac{\theta_1^{-2N_j}\int\limits_{\Xxi}
\bigg\vert \int\limits_{\R} F\circ h_t(x)\cdot \what{\psi}(t/\theta_1^{N_j})\,
dt\bigg\vert^2d\wt{\mu}(x)}{\theta_1^{(N_j-\kappa+1)(2\alpha-2)}\binom{N_j}{N_j-\kappa+1}} \ge 
\\
\int_{\Xxi}
\bigg\vert\int_{\mathbb R}\Phi^+_{pr'(\vec{F}),x}(\tau)(\what{\psi})'(\tau)d\tau\bigg\vert^2 d\widetilde{\mu}(x) - 2\ell C_\psi^2 \delta.
\end{multline*}
 We can choose $\psi$ so that $\what\psi$ is positive and bounded away from zero in some neighborhood of $0$, and also so that the integral in the right-hand side of the last inequality is positive, since $pr'(\vec{F})\ne 0$,  in view of non-degeneracy of the cocycle $\Phi^+_{v,x}$, see  Lemma~\ref{lem-cocycle}(vi) and (v). Thus, choosing $\delta>0$ sufficiently small, we can ensure that the right-hand side is bounded away from zero. Recalling (\ref{eqa0}), we obtain
$$
\liminf\limits_{j\to\infty}\frac{\int\limits_\R|\psi(\theta_1^N\omega)|^2\,d\sigma_F(\omega)}
{(\theta_1^{N_j})^{2\alpha-2}\binom{N_j}{N_j-\kappa+1}} > 0.
$$
The upper bound is even easier, replacing $\liminf$ with $\limsup$ in (\ref{lim1}); in fact, we get an explicit upper bound
$$
\limsup\limits_{j\to\infty}\frac{\int\limits_\R|\psi(\theta_1^N\omega)|^2\,d\sigma_F(\omega)}
{(\theta_1^{N_j})^{2\alpha-2}\binom{N_j}{N_j-\kappa+1}} \le \ell^2 C_\psi^2.
$$
Together, these yield  (\ref{bounds}), in view of $\sup_j(N_{j+1}-N_j) < \infty$.
\end{proof}

\subsection{Constant length case and conclusion of the proof of the main theorem} Now suppose that the primitive aperiodic substitution $\zeta$ has constant length $q$. In this case $\Sf=\Sf_\zeta$ has all column sums equal to $q$, and
so the PF eigenvector of $\Sf^t$ is simply $s=\One$. It follows that the self-similar suspension flow $(\Xxi, h_t, \wt{\mu})$ is the suspension over $(X_\zeta, T_\zeta,\mu)$ with the
constant-one roof function. It turns out that there is a simple relation between the spectral measures. The following lemma is well-known to the experts: it follows, for instance, from considering induced representations, see \cite[Prop.\,1.1]{DL}. Thanks to Mariusz Lema{\'n}czyk who pointed this out to us. For the reader's convenience, we provide a direct elementary proof below, which was given by Nir Lev in a personal communication.

\begin{lemma}\label{lem-equivalence}
Let $(X,T,\mu)$ be an invertible probability-preserving system, and let $(\wt{X}, h_t, \wt{\mu})$ be the suspension flow with the constant-one roof function.
For  $f\in L^2(X,\mu)$ consider the
function $F\in L^2(\wt{X},\wt{\mu})$  defined by
$F(x,\tau)=f(x)$. Then
the following relation holds between the spectral measures $\sigma_F$  on $\R$ and $\sigma_f$ on $\T$:
\[
d\sig_F(\om)=\left(\frac{\sin(\pi \om)}{\pi \om}\right)^2 \cdot d\sig_f(e^{2\pi i \om}),\ \ \om\in \R.
\]
\end{lemma}

\begin{proof} We have $\wt{X} = X\times [0,1]$ and $\wt{\mu} = \mu\times m$ (the Lebesgue measure on $[0,1]$).
The action $h_t$ on $\wt{X}$ is defined by
$h_t(F(x,\tau))=F(T^{\lfloor t+\tau\rfloor}(x),\{t+\tau\})$.
Computation of Fourier coefficients of $F$ yields
\begin{eqnarray*}
\what{\sigma}_F(-t)  =  \langle F\circ h_t, F\rangle & = &
 \int_0^1 \int_{X} h_t(F(x,\tau))\,\ov{F(x,\tau)}\,d\mu(x) \,dm(\tau) \\[1.2ex]   & = &
\int_0^1 \int_{X}f(T^{\lfloor t+\tau\rfloor}(x))\,\ov{f(x)}\, d\mu(x)\, dm(\tau)\\[1.2ex]
& =  & \int_0^{1-\{t\}}\int_{X}f(T^{\lfloor t+\tau\rfloor}(x))\,\ov{f(x)}\, d\mu(x)\, dm(\tau) \\[1.2ex]
 & + & \int_{1-\{t\}}^1\int_{X}f(T^{\lfloor t+\tau\rfloor}(x))\,\ov{f(x)}\, d\mu(x)\, dm(\tau)\\[1.2ex] 
& =  & (1-\{t\})\cdot\what{\sigma}_f(-\lfloor t\rfloor)+\{t\}\cdot\what{\sigma}_f(-\lfloor t+1\rfloor).
\end{eqnarray*}
Thus the Fourier transform of the measure $\sigma_F$ on $\Z$ is the same as the Fourier coefficients of $\sigma_f$, and  $\what\sigma_F$ on $\R$ is obtained by linear interpolation from two adjacent integers.
Consider the ``triangle'' function:
\[
\theta(t)=\begin{cases}0,& t <  -1\text{ or } t > 1,\\ t+1,&-1\le t <0,\\1-t ,& 0\le t \le 1. \end{cases}
\]
Then we have 
\[
\what{\sigma}_F(t)=\sum_{n\in{\mathbb Z}}\what{\sigma}_f(n)\,\theta(t-n).
\]
The inverse Fourier transform of the right-hand side (in the distributional sense) is
\[
\check{\theta}(\om)\cdot \sum_{n\in{\mathbb Z}}\hat{\sigma}_f(n)e^{2\pi i n \om}=\check{\theta}(\om)\,d\sigma_f(\om),
\]
where $\check{\theta}(\om)=(\sin(\pi \om)/(\pi \om))^2$, hence the result.
\end{proof}

\begin{proof}[Proof of Theorem~\ref{th-dimzero}] is  immediate from Proposition~\ref{prop-dimzero} and Lemma~\ref{lem-equivalence}.
\end{proof}

\begin{proof}[Proof of Theorem~\ref{th-main1}] follows from Corollary~\ref{cor-leb} and Theorem~\ref{th-dimzero}.
\end{proof}


\section{Examples and concluding remarks}\label{examples}

We start with a few general comments. Our result concerns constant length substitutions that are primitive and aperiodic.  Aperiodicity may be checked using Pansiot's Lemma \cite{Pansiot}: a primitive constant length substitution $\zeta$ on $\Ak$, which is 1-1 on letters, is aperiodic if and only if there exists $a\in \Ak$ which has at least two distinct ``neighborhoods'' in the set of allowed words, that is distinct words of the form
$bac$ which appear in the sequences $x\in X_\zeta$. The next step is to determine the ergodic classes of the bi-substitution, which is straightforward. Recall that  there is only one ergodic class $E_0= \{(a,a):\ a\in \Ak\}$ if and only if the spectrum of the substitution system is purely discrete. One way to guarantee that this is not the case is to avoid all ``coincidences'' altogether, in other words, make sure that for $a\ne b$ the substituted words $\zeta(a)$ and $\zeta(b)$ differ in every position. Such substitutions are called {\em bijective}. It is clear that a bijective substitution has no transitive pairs $(a,b)$ and at least two ergodic classes. The classical example of a bijective substitution is Thue-Morse: $0\to 01, \ 1\to 10$. A bijective substitution of constant length $q$ on an alphabet of size $m$ is determined by a sequence of $q$ permutations $\phi_1,\ldots,\phi_q$ of the set $\{1,\ldots,m\}$. Following \cite{Queff}, we say that it is {\em bijective abelian} if  these permutations generate a commutative subgroup of 
$S_m$. Clearly, all 2-letter bijective substitutions (also called ``generalized Thue-Morse) are abelian, but most aperiodic bijective substitutions on 3 and more letters are not. It is proved in \cite{Queff} that all bijective abelian substitutions have singular spectrum (actually, it is not explicitly stated there, see \cite[Theorem 3.19]{Bartlett} for a derivation). The 2-letter case was independently proven in \cite{Baake} by a different method.


\begin{example} \label{ex1}
Let $\Ak = \{1,2,3,4\}$ and $\zeta:\ \Ak\to \Ak$ is defined by
$$
\begin{array}{ccc} \zeta(1) & = & 113 \\ \zeta(2) &  = & 232 \\ \zeta(3) & = & 324 \\ \zeta(4) & = & 441 \end{array}
$$
\end{example}

\noindent It is easy to check that this is a primitive, aperiodic substitution of height one, which is bijective non-abelian. 
Its substitution matrix, with its eigenvalues and the corresponding eigenvectors are 
$$
\Sf_\zeta = \left[ \begin{array}{cccc} 2 & 0 & 0 & 1 \\ 0 & 2 & 1 & 0 \\ 1 & 1 & 1 & 0 \\ 0 & 0 & 1 & 2 \end{array}\right],\ \ \{3,2,1,1\},\ \ \mbox{eigenvectors:} \ \ 
\left[ \begin{array}{r} 1 \\ 1\\ 1 \\1 \end{array} \right],\ \left[ \begin{array}{r} 1 \\ -1\\ 0  \\ 0  \end{array} \right],\ \left[ \begin{array}{r} 1 \\ 1\\ -1 \\ -1 \end{array} \right],\ \left[ \begin{array}{r} 1 \\ -1\\ 1 \\ -1 \end{array} \right].
$$
Thus, according to Theorem~\ref{th-main1}, the spectrum of the measure-preserving system $(X_\zeta,T_\zeta,\mu)$ is singular. We continue with a more detailed analysis, in order to illustrate the concepts and the procedure described in Section~\ref{Queffelec}.

There are two ergodic classes for the bi-substitution $\zeta^{[2]}: \ E_0 =\{(a,a):\ a\in \Ak\}$ and $E_1=\{(a,b):\ a\ne b\}$. (By the way, note that having two ergodic classes for a bijective substitution is equivalent to the condition that the group generated by the permutations $\phi_j$ acts transitively on the set of pairs $\Ak\times \Ak$.)  Vectors $(v_{ab}) \in F\subset \C^{\Ak\times \Ak}$, the 2-dimensional
eigenspace of the transpose of the coincidence matrix ${\sf C}^t$, corresponding to the maximal eigenvalue $q=3$, are written in matrix form
$$
W = \left[ \begin{array}{cccc} w_0 & w_1 & w_1 & w_1 \\ w_1 & w_0 & w_1 & w_1 \\ w_1 & w_1 & w_0 & w_1 \\ w_1 & w_1 & w_1 & w_0 \end{array}\right].
$$
The set $F_+$ consists of those $v$ for which the above matrix $W$ is semi-positive. It is easy to see that it has eigenvalues $w_0 + 3w_1$ and $w_0-w_1$, the latter with multiplicity 3. Thus,
$$
\Qk = \{v\in F_+:\ v_{aa}=1,\ a\in \Ak\} \sim \{W:\ w_0= 1,\ -1/3 \le w_1 \le 1\},
$$ 
with extreme points $v^{(0)} \sim (w_0,w_1) = (1,1)$ and $v^{(1)}\sim (1,-1/3)$. According to Queff\'elec's Theorem \ref{th-queff}, the maximal spectral type of the substitution system is $(\lam_0 + \lam_1) * \om$, where
$\lam_i = v^{(i)} \Sig$. As usual, $\lam_0 = \delta_0$, so we focus on $\lam_1$, which is given by
$$
\lam_1 = \begin{bmatrix}
1&-1/3&-1/3&-1/3\\
-1/3&1&-1/3&-1/3\\
-1/3&-1/3&1&-1/3\\
-1/3&-1/3&-1/3&1
\end{bmatrix} \cdot (\sig_{ij})_{i,j=1}^4 = \begin{array}{r} \sigma_1+\sigma_2+\sigma_3+\sigma_4\\ -\sigma_{12}/3-\sigma_{21}/3-
\sigma_{13}/3-\sigma_{31}/3\\ -\sigma_{14}/3 -\sigma_{41}/3-\sigma_{23}/3-
\sigma_{32}/3\\ -\sigma_{24}/3 -\sigma_{42}/3-\sigma_{34}/3-\sigma_{43}/3 \end{array}
$$
Note that in the formula above the product is a linear combination of spectral measures with coefficients taken from the matrix, rather than matrix multiplication.
Diagonalizing the corresponding quadratic form, as in Proposition~\ref{decomposition}, we express $\lam_1$ as a positive linear combination of spectral measures of cylindrical functions orthogonal to the constants:
$$
\lam_1 = \frac{1}{3} \bigl[\sig_{\One_1 - \One_2} + \sig_{\One_3 - \One_4} + \sig_{\One_1 + \One_2 - \One_3 -\One_4}\bigr].
$$
We can apply Proposition~\ref{prop-dimzero} to obtain, noting that $\theta_1=3,\ \theta_2 = 2,\ j=2$, and $\kappa=1$:
$$
\lam_1(B_r(0)) \asymp r^{2-2\log_3 2}.
$$

The next example shows that the spectrum of substitution matrix eigenvalues can change
after transition to the pure base substitution (as we saw in Proposition~\ref{prop-reduce}, the difference can only be in roots of unity and zero eigenvalues).

\begin{example}\label{addeigen}
Let $\zeta:1\to 142, 2 \to 253, 3\to 251, 4 \to 514, 5 \to 415$. 
\end{example}
The substitution $\zeta$ is primitive and aperiodic, its height equals 2. The alphabet $\mathcal I$
of the pure base substitution consists of letters $a,\ b,\ c,\ d,\ e,\ f$ that correspond to the following sequences of length 2:
$a\leftrightarrow 14,\ b\leftrightarrow 24,\ c\leftrightarrow 34,\ d\leftrightarrow 25,\ e\leftrightarrow 35,
\ f\leftrightarrow 15$. The pure base substitution $\eta\colon
a \to ada, b\to dea, c \to dfa, d \to dcf, e \to daf, f \to abf$.
\[
S_\zeta=
\begin{bmatrix}
1&0&1&1&1\\
1&1&1&0&0\\
0&1&0&0&0\\
1&0&0&1&1\\
0&1&1&1&1
\end{bmatrix},
S_\eta=
\begin{bmatrix}
2&1&1&0&1&1\\
2&1&0&1&0&1\\
1&1&2&0&0&0\\
1&0&2&0&1&0\\
0&1&0&0&0&0\\
0&0&1&1&1&1\\
\end{bmatrix}.
\]
 The eigenvalues of $S_\zeta$ are $\{3,1,0,0,0\}$ and those of $S_\eta$ are $\{3, (1+i\sqrt{3})/2, (1-i\sqrt{3})/2,0,0,0\}$. 
The  substitution $\eta$ has a coincidence, hence the spectrum is pure discrete by Dekking's Theorem.

\begin{example} {\em (\cite[Example 9.3 and Section 11.1.2.3]{Queff})} \label{ex2} Let
$\zeta:\ 0\to 001,\ 1\to 122,\ 2\to 210$.
\end{example}

It is a bijective non-abelian substitution (also primitive, aperiodic, and of height one). The eigenvalues of the substitution matrix are $\{3,1,0\}$, so we obtain singularity of the spectrum already from Corollary~\ref{cor-half}. In \cite{Queff} it was erroneously stated that this system has a Lebesgue component in its spectrum; this was corrected by Bartlett \cite{Bartlett} with the help of his algorithm computing the Fourier coefficients of spectral measures.

\begin{example} {\em (Rudin-Shapiro)}
Let $\zeta:\ 0\to 01,\ 1\to 02,\ 2\to 31,\ 3\to 32$.
\end{example}

This is a non-bijective substitution; the bi-substitution has two ergodic classes and a transitive class. It has a Lebesgue spectral component \cite{Kaku}, see also \cite{Queff}. The eigenvalues of the substitution matrix are $\{2,\pm \sqrt{2},0\}$.

\medskip

On the other hand, consider the following, ``modified Rudin-Shapiro-like'' substitution:

\begin{example} {\em (\cite{ChGr1})}
Let $\zeta:\ 0\to 01,\ 1\to 20,\ 2\to 13,\ 3\to 32$.
\end{example}

It is bijective; the bi-substitution has three ergodic classes (the same as for Rudin-Shapiro, only the transitive class now becomes an ergodic class). The eigenvalues of the substitution matrix are again $\{2,\pm \sqrt{2},0\}$; however, here the spectrum is singular, as shown by Chan and Grimm \cite{ChGr1}, with the help of Bartlett's algorithm. Thus, our sufficient condition for singularity in Theorem~\ref{th-main1} is not a necessary condition.

\subsection{Open Questions}

\medskip
\renewcommand{\theenumi}{\roman{enumi}}
\begin{enumerate}
\item Is it true that if the constant length substitution is bijective, then the spectrum of the dynamical system is singular?

\medskip

The next two questions concern {\em non-constant length} substitutions.

\medskip

\item As far as we are aware, all the known examples of substitution systems with a Lebesgue spectral component are either constant length or their simple ``recodings'', which do not change the spectrum of the substitution matrix in an essential way. Is there a substitution with a Lebesgue spectral component, such that the PF eigenvalue $\theta_1$ of the substitution matrix is not an integer?

\medskip

\item Is there a result similar to Theorem~\ref{th-main1} for non-constant length substitutions, namely, that the absence of an eigenvalue of absolute value $\sqrt{\theta_1}$ for the substitution matrix implies singularity? Perhaps, there is at least an analog of Corollary~\ref{cor-half}, namely, that $|\theta_2| < \sqrt{\theta_1}$ implies singularity?
\end{enumerate}
\medskip



{\bf Acknowledgement:} We are grateful to Fabien Durand and Nir Lev for helpful discussions. Thanks also to Michael Baake, Franz G\"ahler, and Mariusz Lema{\'n}czyk for their comments on the first version of the paper.


\end{document}